\documentclass[12pt,dvips]{amsart}
\usepackage{amsfonts, amssymb, latexsym, epsfig}
\usepackage{euler, amsfonts, amssymb, latexsym, epsfig, epic}

\newtheorem{Theorem}{Theorem}[section]

\newtheorem{Proposition}[Theorem]{Proposition} 
\newtheorem*{Main Theorem*}{Main Theorem} 
\newtheorem{Lemma}[Theorem]{Lemma}

\newtheorem{Corollary}[Theorem]{Corollary}

\newtheorem{Definition-Lemma}[Theorem]{Definition-Lemma}
\newtheorem{Main Conjecture}[Theorem]{Main Conjecture}
\theoremstyle{remark}
\newtheorem{Example}[Theorem]{Example}

\newcommand\evac{{\tt evac}}
\newcommand\infusion{{\tt infusion}}

\theoremstyle{plain}

\newtheorem*{MainThm}{Main Theorem}

\newcommand{\cellsize}{12}
\newlength{\cellsz} 
\newsavebox{\cell}
\sbox{\cell}{\begin{picture}(\cellsize,\cellsize)
\put(0,0){\line(1,0){\cellsize}}
\put(0,0){\line(0,1){\cellsize}}
\put(\cellsize,0){\line(0,1){\cellsize}}
\put(0,\cellsize){\line(1,0){\cellsize}}
\end{picture}}
\newcommand\cellify[1]{\def\thearg{#1}\def\nothing{}%
\ifx\thearg\nothing
\vrule width0pt height\cellsz depth0pt\else
\hbox to 0pt{\usebox{\cell} \hss}\fi%
\vbox to \cellsz{
\vss
\hbox to \cellsz{\hss$#1$\hss}
\vss}}
\newcommand\tableau[1]{\vtop{\let\\\cr
\baselineskip -16000pt \lineskiplimit 16000pt \lineskip 0pt
\ialign{&\cellify{##}\cr#1\crcr}}}
\begin{document}
\pagestyle{plain}

\mbox{}\vspace{-2.0ex}
\title{An $S_3$-symmetric Littlewood-Richardson rule}
\author{Hugh Thomas}
\address{Department of Mathematics and Statistics, University of New Brunswick, Fredericton, New Brunswick, E3B 5A3, Canada }
\email{hugh@math.unb.ca}

\author{Alexander Yong}
\address{Department of Mathematics, University of Minnesota, Minneapolis, MN 55455, USA}

\email{ayong@math.umn.edu}

\date{April 4, 2007}

\maketitle
\section{Introduction}
Fix Young shapes $\lambda,\mu,\nu\subseteq \Lambda:=\ell\times k$.
Viewing the {\bf Littlewood-Richardson coefficients} 
$C_{\lambda,\mu,\nu}$
as intersection numbers of three
Schubert varieties in general position in the Grassmannian of 
$\ell$-dimensional planes
in ${\mathbb C}^{\ell+k}$, one has the obvious $S_3$-symmetries: 
\begin{equation}
\label{eqn:basicsym}
C_{\lambda,\mu,\nu}=C_{\mu,\nu,\lambda}=C_{\nu,\lambda,\mu}
=C_{\mu,\lambda,\nu}=C_{\nu,\mu,\lambda}=C_{\lambda,\nu,\mu}.
\end{equation}
There is interest in the combinatorics of these symmetries;
previously studied Littlewood-Richardson rules for $C_{\lambda,\mu,\nu}$
manifest at most three of the six, see, e.g.,
\cite{Bere.Zele, Knutson.Tao.Woodward, Vall.Pak, Kamnitzer} and the references
therein. We construct a {\bf carton rule} for $C_{\lambda,\mu,\nu}$ 
that transparently and uniformly explains all symmetries (\ref{eqn:basicsym}).

Figure~\ref{fig:growth} depicts a 
{\bf carton}, i.e., a three-dimensional box with 
a grid drawn rectilinearly on the six faces of its surface. Along the 
``$\emptyset,T_\lambda, T_\mu$'' face, the 
grid has $(|\mu|+1)\times (|\lambda|+1)$
vertices (including ones on the bounding edges of the face); the remainder
of the grid is similarly determined.
Define a {\bf carton filling} to be an assignment of a Young diagram 
to each vertex of the grid so the shapes increase
one box at a time while moving away from $\emptyset$, and so that, 
for any subgrid
$\begin{matrix}
\alpha & - & \beta\\
| & & | \\
\gamma & - &\delta
\end{matrix}$ the {\bf Fomin growth assumptions} hold:
\begin{itemize}
\item[(F1)] if $\alpha$ is the unique shape containing $\gamma$ and
contained in $\beta$, then $\delta=\alpha$;
\item[(F2)] otherwise there is a unique such shape other than $\alpha$, 
and this shape is $\delta$. 
\end{itemize}
Notice that these conditions are symmetric in $\alpha$ and $\delta$.

Initially, assign the shapes 
$\emptyset$ and $\Lambda$ to opposite corners, as in Figure 1. A 
standard Young tableau $T\in SYT(\sigma/\pi)$ of shape 
${\sigma/\pi}$ is equivalent to a 
{\bf shape chain} in Young's lattice, e.g.,
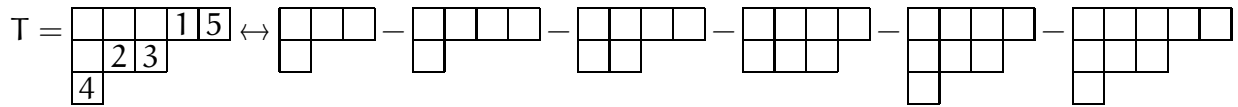
\[T=\tableau{{\ }&{\ }&{ \ }&{1}&{5}\\{\ }&{2}&{3}\\{4}}
\leftrightarrow 
\tableau{{\ }&{\ }&{\ }\\{\ }}-
\tableau{{\ }&{\ }&{\ }&{\ }\\{\ }}-
\tableau{{\ }&{\ }&{\ }&{\ }\\{\ }&{\ }}-
\tableau{{\ }&{\ }&{\ }&{\ }\\{\ }&{\ }&{\ }}-
\tableau{{\ }&{\ }&{\ }&{\ }\\{\ }&{\ }&{\ }\\{\ }}-
\tableau{{\ }&{\ }&{\ }&{\ }&{\ }\\{\ }&{\ }&{\ }\\{\ }}\]
by starting with the unfilled inner shape, and adding one box at a time to
indicate the box containing $1$, then $2$, etc. 
Fix a choice of standard tableaux $T_{\lambda},T_{\mu}$ and
$T_{\nu}$ of respective shapes $\lambda,\mu$ and $\nu$. 
Initialize the indicated edges 
with the shape chains for these tableaux.

Let ${\tt CARTONS}_{\lambda,\mu,\nu}$ be
all carton fillings with the above 
initial data.

\begin{MainThm}
\label{thm:main}
The Littlewood-Richardson coefficient $C_{\lambda,\mu,\nu}$
equals $\#{\tt CARTONS}_{\lambda,\mu,\nu}$. 
\end{MainThm}

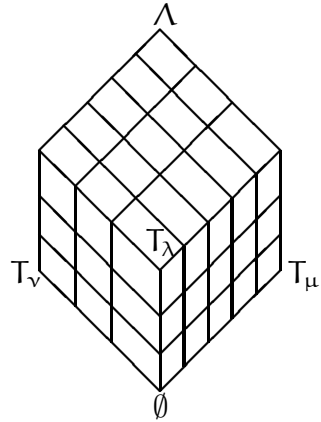
\begin{figure}[h]
\unitlength=.32mm
\begin{picture}(120,160)
\thicklines
\put(60,10){\line(1,1){50}}
\put(60,25){\line(1,1){50}}
\put(60,41){\line(1,1){50}}
\put(60,10){\line(-1,1){50}}
\put(60,10){\line(0,1){50}}
\put(70,20){\line(0,1){50}}
\put(80,30){\line(0,1){50}}
\put(90,40){\line(0,1){50}}
\put(100,50){\line(0,1){50}}
\put(10,60){\line(0,1){50}}
\put(25,45){\line(0,1){50}}
\put(40,30){\line(0,1){50}}

\put(110,60){\line(0,1){50}}
\put(10,110){\line(1,1){50}}
\put(25,95){\line(1,1){50}}
\put(40,80){\line(1,1){50}}

\put(110,110){\line(-1,1){50}}
\put(100,100){\line(-1,1){50}}
\put(90,90){\line(-1,1){50}}
\put(80,80){\line(-1,1){50}}
\put(70,70){\line(-1,1){50}}

\put(60,60){\line(1,1){50}}
\put(60,60){\line(-1,1){50}}
\put(60,41){\line(-1,1){50}}
\put(60,25){\line(-1,1){50}}

\put(57,-1){$\emptyset$}
\put(113,55){$T_\mu$}
\put(54,68){$T_\lambda$}
\put(-2,55){$T_\nu$}
\put(57,162){$\Lambda$}
\end{picture}
\caption{\label{fig:growth} The carton rule 
counts $C_{\lambda,\mu,\nu}$ by
assigning Young diagrams to the vertices of the six faces}
\end{figure}

This rule manifests bijections between 
${\tt CARTONS}_{\lambda,\mu,\nu}$
and ${\tt CARTONS}_{\delta,\epsilon,\zeta}$ for permutations
$(\delta,\epsilon,\zeta)$ of $(\lambda,\mu,\nu)$. We remark
that it can be stated as counting lattice points
of a $0,1$-polytope, and readily extends to the setting of 
\cite{Thomas.Yong, Thomas.Yong:DE}.

 The Theorem is proved in Section~2, starting from the
jeu de taquin formulation of the Littlewood-Richardson rule, and
using Fomin's growth diagram ideas. In Figure~\ref{fig:easyex} we give an example of the Main Theorem. An extended example
is given in Section~3.

\begin{figure}[h]
\unitlength=.33mm
\begin{picture}(200,360)
\thicklines
\put(80,30){\line(1,1){150}}
\put(80,30){\line(-1,1){50}}
\put(80,30){\line(0,1){100}}
\put(80,130){\line(1,1){150}}
\put(230,180){\line(0,1){100}}
\put(30,80){\line(0,1){100}}
\put(30,180){\line(1,1){150}}
\put(30,180){\line(1,-1){50}}
\put(180,330){\line(1,-1){50}}
\put(80,80){\line(1,1){150}}
\put(80,80){\line(-1,1){50}}
\put(130,80){\line(0,1){100}}
\put(130,180){\line(-1,1){50}}
\put(180,130){\line(0,1){100}}
\put(180,230){\line(-1,1){50}}
\put(77,10){$\emptyset$}
\put(130,60){$\tableau{{\ }}$}
\put(180,110){$\tableau{{\ }&{\ }}$}
\put(235,160){$\tableau{{1 } &{2 }\\{3 }}=\mu$}
\put(235,230){$\tableau{{\ }&{ \ }\\{ \ }&{\ }}$}
\put(235,280){$\tableau{{\ }&{\ }&{\ }\\{ \ }&{ \ }}$}
\put(85,70){$\tableau{{\ }}$}
\put(135,120){$\tableau{{\ }\\{\ }}$}
\put(185,170){$\tableau{{\ }&{\ }\\{\ }}$}
\put(85,120){$\tableau{{1}&{2}}$}
\put(85,105){$=\lambda$}
\put(135,170){$\tableau{{\ }&{ \ }\\{\ }}$}
\put(185,215){$\tableau{{\ }&{ \ }&{ \ }\\{ \ }}$}
\put(-12,70){$\nu=\tableau{{1 }}$}
\put(10,120){$\tableau{{ \ }\\{\ }}$}
\put(0,170){$\tableau{{\ }&{\ }\\{\ }}$}
\put(55,245){$\tableau{{\ }&{\ }\\{ \ }&{\ }}$}
\put(105,295){$\tableau{{\ }&{\ }&{\ }\\{\ }&{\ }}$}
\put(155,345){$\tableau{{\ }&{\ }&{ \ }\\{\ }&{ \ }&{ \ }}=\Lambda$}
\end{picture}
\caption{\label{fig:easyex} The ``front'' three faces of a
carton filling for $C_{(2),(2,1),(1)}=1$. The choices of
tableaux $T_{\lambda}, T_{\mu}$ and
$T_{\nu}$ are as shown.}
\end{figure}
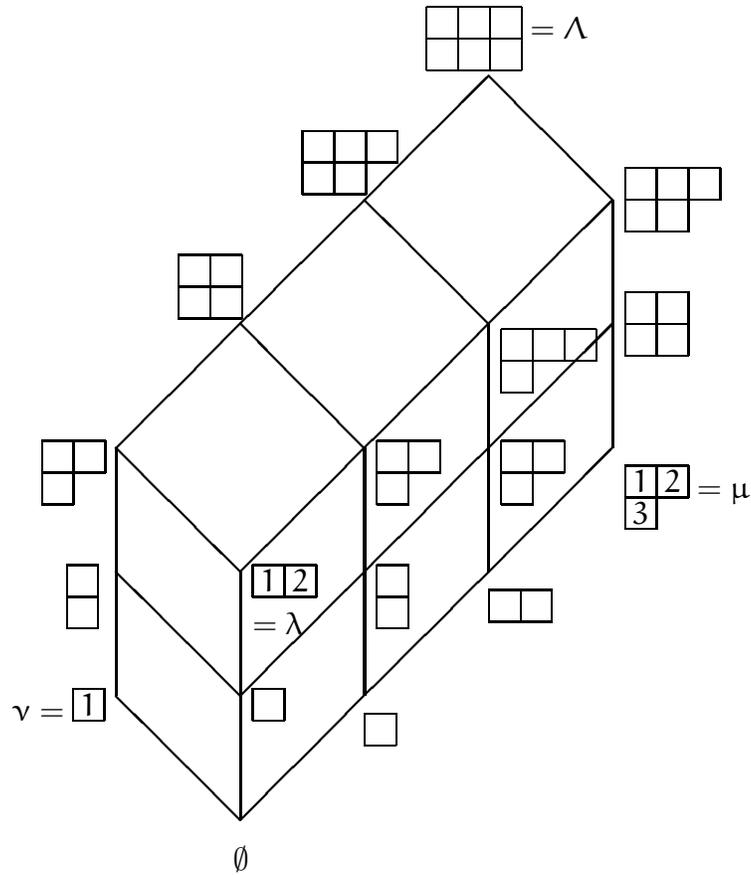

\section{Tableau facts and proof of the Main Theorem}

\subsection{Tableau sliding}
There is a partial order $\prec$ 
on the rectangle $\Lambda$ where
$x\prec y$ if $x$ is weakly northwest of $y$.  
Given $T\in {\rm SYT}(\nu/\lambda)$ 
consider $x\in \lambda$, maximal in $\prec$ subject to being
less than \emph{some} box of $\nu/\lambda$. 
Associate another standard tableau ${\tt jdt}_{x}(T)$, called the {\bf jeu de
taquin slide} of $T$ into $x$: Let $y$ be the box of $\nu/\lambda$
with the smallest label, among those covering $x$. 
Move the label of 
$y$ to $x$, leaving
$y$ vacant. Look for boxes of $\nu/\lambda$ covering $y$ and repeat,
moving into $y$ the smallest label among those boxes covering it. 
Then ${\tt jdt}_{x}(T)$ results when no further 
slides are possible. The {\bf rectification} of $T$
is the iteration of jeu de taquin slides until 
terminating at a straight shape standard tableau ${\tt rectification}(T)$.

We can impose a specific
``inner'' $U\in {\rm SYT}(\lambda)$ that encodes the order
in which the jeu de taquin slides are done. For example, if $U=\tableau{{1}&{2}&{3}\\{4}}$ then
\[T=\tableau{{\ }&{\ }&{ \ }&{1}&{5}\\{\ }&{2}&{3}\\{4}&{6}}
\to \tableau{{\ }&{\ }&{ \ }&{1}&{5}\\{2 }&{3}&{\ }\\{4}&{6}}
\to \tableau{{\ }&{\ }&{ 1 }&{5}&{\ }\\{2 }&{3}&{\ }&\\{4}&{6}}
\to \tableau{{\ }&{1 }&{ 5 }&{\ }&{\ }\\{2 }&{3}&{\ }\\{4}&{6}}
\to \tableau{{1 }&{3 }&{ 5 }&{\ }&{\ }\\{2 }&{6}&{\ }\\{4}&{\ }}
\]
$={\tt rectification}(T)$. 
Placing these chains
one atop another (starting with $T$'s and ending with
${\tt rectification}(T)$'s) gives a Fomin {\bf growth diagram}, which
in this case is given in Table~\ref{table:growth}. 
Growth diagrams satisfy (F1) and (F2); see Fomin's~\cite[Appendix~2]{Stanley}
for more.

\begin{table}[h]
\begin{center}
\begin{tabular}{|l|l|l|l|l|l|l|}
\hline
$(3,1)$ & $(4,1)$ & $(4,2)$ & $(4,3)$ & $(4,3,1)$ & $(5,3,1)$ & $(5,3,2)$\\ \hline
$(3)$ & $(4)$ & $(4,1)$ & $(4,2)$ & $(4,2,1)$ & $(5,2,1)$ & $(5,2,2)$\\ \hline
$(2)$ & $(3)$ & $(3,1)$ & $(3,2)$ & $(3,2,1)$ & $(4,2,1)$ & $(4,2,2)$\\ \hline
$(1)$ & $(2)$ & $(2,1)$ & $(2,2)$ & $(2,2,1)$ & $(3,2,1)$ & $(3,2,2)$\\ \hline
$\emptyset$ & $(1)$ & $(1,1)$ & $(2,1)$ & $(2,1,1)$ & $(3,1,1)$ & $(3,2,1)$\\ \hline
\end{tabular}
\end{center}
\caption{\label{table:growth} A growth diagram}
\end{table}

Given a top row $U$ and left column
$T$, define ${\tt infusion}_1(U,T)$ to be the bottom row of the growth
diagram, and ${\tt infusion}_2(U,T)$ to be the right column.  
We write ${\tt infusion}(U,T)$ for the ordered pair 
$({\tt infusion}_1(U,T),{\tt infusion}_2(U,T))$. 
Growth diagrams are transpose symmetric, because (F1) and (F2) are. So 
one has the {\bf infusion involution}: 
\[{\tt infusion}({\tt infusion}(U,T))=(U,T).\]
If $U$ is a straight shape, ${\tt infusion}_1(U,T)=
{\tt rectification}(T)$, while ${\tt infusion}_2(U,T)$ encodes the order
in which squares were vacated in the jeu de taquin process.  

Also, given 
$T\in {\rm SYT}(\nu/\lambda)$, consider $x\in\Lambda\setminus
\nu$ \emph{minimal} subject to being larger than
some element of $\nu/\lambda$. The {\bf reverse jeu de taquin
slide} ${\tt revjdt}_{x}(T)$ of $T$ into $x$ is defined similarly
to a jeu de taquin slide, except we move into $x$ the \emph{largest} of
the labels among boxes in $\nu/\lambda$ covered by~$x$. 
We define {\bf reverse rectification} ${\tt revrectification}(T)$ similarly.
The first fundamental theorem of jeu de taquin asserts 
{\tt (rev)rectification} is well-defined.

Fix $T_{\mu}\in {\rm SYT}(\mu)$. The
number of $T\in {\rm SYT}(\nu/\lambda)$ such that ${\tt rectification}(T)=T_{\mu}$
equals $C_{\lambda,\mu,\nu^{\vee}}=C_{\lambda,\mu}^{\nu}$ where $\nu^{\vee}$ is the
straight shape obtained as the $180$-degree rotation of $\Lambda\setminus \nu$. This
is Sch\"{u}tzenberger's jeu de taquin formulation of the Littlewood-Richardson rule, 
see~\cite[Appendix~2]{Stanley}.

By a {\bf slide} we mean either kind of jeu de taquin slide. 
Consider two equivalence relations on a 
pair of tableaux $T$ and $U$. Tableaux are 
{\bf jeu de taquin equivalent} if one
can be obtained from the other by a sequence of slides.
They are {\bf dual equivalent} if any such sequence results in tableaux
of the same shape \cite{Haiman:DE}. Facts: Tableaux of the same
straight shape are dual equivalent. A common application of
slides to dual equivalent tableaux
produces dual equivalent tableaux.
A pair of tableaux 
that are both jeu de taquin and dual equivalent must be equal.

Recall Sch\"{u}tzenberger's evacuation map. 
For $T\in {\rm SYT}(\lambda)$, let 
${\hat T}$ be obtained by erasing the entry $1$ (in the northwest 
corner $c$) of $T$
and subtracting $1$ from the remaining entries. Let
$\Delta(T)={\tt jdt}_{c}({\hat T})$.
The {\bf evacuation} ${\tt evac}(T)\in {\rm SYT}(\lambda)$ 
is defined by the shape chain
\[\emptyset={\tt shape}(\Delta^{|\lambda|}(T))-
{\tt shape}(\Delta^{|\lambda|-1}(T))-\ldots - {\tt shape}(\Delta^{1}(T)) - T.\]
This map is an involution: 
${\tt evac}({\tt evac}(T))=T$.

\subsection{Proof of the rule} 
Given $T\!\in\! {\rm SYT}(\alpha)$ for a straight shape $\alpha$, define
${\widetilde T}\!\in\! {\rm SYT}({\tt rotate}(\alpha))$ where ${\tt rotate}(\alpha)=\Lambda\setminus\alpha^{\vee}$ by computing 
${\tt evac}(T)\in {\rm SYT}(\alpha)$, replacing entry $i$ with $|\alpha|-i+1$
throughout and rotating the resulting tableau $180$-degrees and placing it 
at the bottom right corner of $\Lambda$.

We need the following well-known fact.  The proof we give extends
straightforwardly to the setup of \cite{Thomas.Yong,Thomas.Yong:DE}, allowing a ``cominuscule'' version of the Main Theorem.
\begin{Lemma}
Suppose $\alpha,\beta$ and $\gamma$ are shapes where
\[T_{\beta}\in {\rm SYT}(\beta), T_{\gamma^{\vee}/\alpha}\in {\rm SYT}(\gamma^{\vee}/\alpha)
\mbox{ and } {\tt rectification}(T_{\gamma^{\vee}/\alpha})=T_{\beta}.\]
Then ${\tt revrectification}(T_{\gamma^{\vee}/\alpha})={\tt revrectification}(T_{\beta})={\widetilde T}_{\beta}$.
\end{Lemma}
\begin{proof}
It suffices to show that ${\tt revrectification}(T_{\beta})={\widetilde
T}_\beta$. We induct on $|\beta|=n$.  

Let $U$ be the tableau obtained by removing the box labeled $1$ from 
$T_\beta$.  So $U$ is of skew shape $\beta/(1)$. 
Let $V={\tt rectification}( U)$, a tableau of shape $\kappa\subset \beta$.
By induction, ${\tt revrectification}(V)=\widetilde{V}$.
Now ${\tt revrectification}(T_{\beta})$ and ${\tt revrectification}(V)$ agree 
except that the former has a label 1 in a box that does not appear
in the latter.
In \cite[Proposition 4.6]{Thomas.Yong}, 
we showed that the reverse rectification of a tableau of
shape $\beta$ is necessarily of shape ${\tt rotate}(\beta)$.  
Thus,
the location of 1 in ${\tt revrectification} (T_{\beta})$ must be the box
$\kappa^\vee/\beta^\vee$.  This is the $180$-degree
rotation of the box $\beta/\kappa$ in $\Lambda$, the latter being
the position of $n$ in
${\tt evac}(T_{\beta})$.  Thus, $1$ is located in the desired position in
${\tt revrectification}(T_{\beta})$.
The rest of the statement follows from the
fact that
the rest of ${\tt revrectification}(T_{\beta})$ agrees with ${\widetilde V}$, 
and the entries
of ${\tt evac}(T_{\beta})$ other than $n$ agree with ${\tt evac}(V)$.
\end{proof}

\begin{Corollary}
\label{cor:moreinit}
Fix a carton filling. The uninitialized corners of the 
``$\emptyset,T_{\lambda},T_{\mu}$'' and
``$\emptyset,T_{\nu},T_{\lambda}$'' faces are $\nu^{\vee}$ 
and $\mu^{\vee}$ respectively. The remaining ``sixth'' corner (the unique 
one not visible in
Figure~\ref{fig:growth})
is
$\lambda^{\vee}$. Thus, we can speak of the
edges $\lambda^{\vee}-\Lambda$, $\mu^{\vee}-\Lambda$ and $\nu^{\vee}-\Lambda$.
These 
represent the shape chains of ${\widetilde T}_{\lambda}, {\widetilde T}_{\mu}$
and ${\widetilde T}_{\nu}$ respectively.
\end{Corollary}
\begin{proof}
Think of the union of the faces 
``$\emptyset,T_{\mu},T_{\lambda}$'' 
with the adjacent face involving $\Lambda$ and the ``sixth
corner''
as a single growth diagram. This diagram computes 
${\tt revrectification}(T_{\lambda})$ and records the result
along the edge involving $\Lambda$ and 
diagonally opposite to the $T_{\lambda}$ edge.
But the lemma asserts this result is ${\widetilde T}_{\lambda}$ and hence
the ``sixth corner vertex'' is assigned $\lambda^{\vee}$.
The other conclusions are proved
similarly.
\end{proof}

Corollary~\ref{cor:moreinit} affords us the 
convenience of referring to a
face by its corner vertices. 
Note any carton filling
gives a growth diagram on the face 
$\emptyset-\mu-\nu^{\vee}-\lambda$ for which the edge $\lambda-\nu^{\vee}$
is a standard tableau of shape $\nu^{\vee}/\lambda$ rectifying
to $T_{\mu}$. By the jeu de taquin Littlewood-Richardson
rule, fillings of this face count $C_{\lambda,\mu,\nu}$. Hence it suffices
to show that any such growth diagram for this face 
extends uniquely to a filling of the entire carton.

If such an extension exists, it is unique: $\lambda-\nu^{\vee}$
and $\Lambda-\nu^{\vee}$ determine, by (F1), (F2) and the Corollary,
the face $\lambda-\mu^{\vee}-\Lambda-\nu^{\vee}$. Similarly, these
two faces determine the remaining faces (in order):
\[\emptyset-\nu-\mu^{\vee}-\lambda\implies 
\nu-\lambda^{\vee}-\Lambda-\mu^{\vee}\implies
\mu-\lambda^{\vee}-\Lambda-\nu^{\vee}\implies
\emptyset-\nu-\lambda^{\vee}-\mu.\]

We now show that 
with a filling of $\emptyset-\mu-\nu^{\vee}-\lambda$ (together
with the extra edges given in the Corollary) one can
extend it to fill the carton, using the conclusions
of Corollary~\ref{cor:moreinit}.

\medskip
\noindent
\underline{$\emptyset-\mu-\nu^{\vee}-\lambda$:}
This is given by labeling $\lambda-\nu^{\vee}$ with 
$T_{\nu^{\vee}/\lambda}\in {\rm SYT}(\nu^{\vee}/\lambda)$ that {\bf witnesses}
$C_{\lambda,\mu,\nu}=C_{\lambda,\mu}^{\nu^{\vee}}$, 
i.e., it rectifies to $T_{\mu}$. By the 
infusion involution, the edge $\mu-\nu^{\vee}$ representing  
$T_{\nu^{\vee}/\mu}:={\tt infusion}_2(T_{\lambda},T_{\nu^{\vee}/\lambda})
\in {\rm SYT}(\nu^{\vee}/\mu)$ witnesses $C_{\mu,\lambda}^{\nu^{\vee}}$.

\medskip
\noindent
\underline{$\lambda-\mu^{\vee}-\Lambda-\nu^{\vee}$:} Build
the growth diagram using $T_{\nu^{\vee}/\lambda}$ and
${\widetilde T}_{\nu}$ from the edges $\lambda-\nu^{\vee}$ and
$\nu^{\vee}-\Lambda$ respectively. 
The only boundary condition we need to check
is that \linebreak ${\tt infusion}_2(T_{\nu^{\vee}/\lambda},{\widetilde T}_{\nu})
={\widetilde T}_{\mu}$, which is true by the Lemma. The newly determined
edge 
$\lambda-\mu^{\vee}$ is 
${\tt infusion}_{1}(T_{\nu^{\vee}/\lambda},{\widetilde T}_{\nu})$. This is a tableau $T_{\mu^{\vee}/\lambda}\in {\rm SYT}(\mu^{\vee}/\lambda)$ which rectifies to $T_{\nu}$, i.e., one that witnesses $C_{\lambda,\nu}^{\mu^{\vee}}$. 

\medskip
\noindent
\underline{$\emptyset-\nu-\mu^{\vee}-\lambda$:} Using the determined edges 
$\lambda-\mu^{\vee}$ and $\emptyset-\lambda$ we obtain a growth diagram
with edge $\emptyset-\nu$ representing 
${\tt infusion}_1(T_{\lambda},T_{\mu^{\vee}/\lambda})$, which equals
$T_{\nu}$, as desired. By the infusion involution, $\nu-\mu^{\vee}$
represents a tableau $T_{\mu^{\vee}/\nu}\in {\rm SYT}({\mu^{\vee}/\nu})$
that witnesses $C_{\nu,\lambda}^{\mu^{\vee}}$. 

\medskip
\noindent
\underline{$\nu-\lambda^{\vee}-\Lambda-\mu^{\vee}$:} This time, we 
grow the face using the edges $\nu-\mu^{\vee}$ and $\mu^{\vee}-\Lambda$.
The edge $\lambda^{\vee}-\Lambda$ is thus 
${\tt infusion}_2(T_{\mu^{\vee}/\nu},{\widetilde T}_{\mu})$ which
indeed equals ${\widetilde T}_{\lambda}$, by the Lemma. 
The newly determined edge $\nu-\lambda^{\vee}$ is 
$T_{\lambda^{\vee}/\nu}:={\tt infusion}_1(T_{\mu^{\vee}/\nu},
{\widetilde T}_{\mu})\in {\rm SYT}(\lambda^{\vee}/\nu)$ that witnesses 
$C_{\nu,\mu}^{\lambda^{\vee}}$. 

\medskip
\noindent
\underline{$\mu-\lambda^{\vee}-\Lambda-\nu^{\vee}$:} Growing the face using
$\mu-\nu^{\vee}$ and $\nu^{\vee}-\Lambda$ we find that the edge 
$\lambda^{\vee}-\Lambda$ equals ${\tt infusion}_2(T_{\nu^{\vee}/\mu},{\widetilde T}_{\nu})$, which is the already determined ${\widetilde T}_{\lambda}$.
The newly determined edge $\mu-\lambda^{\vee}$ is 
$T_{\lambda^{\vee}/\mu}:={\tt infusion}_1(T_{\nu^{\vee}/\mu},{\widetilde T}_{\nu})\in {\rm SYT}(\lambda^{\vee}/\mu)$ that witnesses 
$C_{\mu,\nu}^{\lambda^{\vee}}$. 

\medskip
\noindent
\underline{$\emptyset-\nu-\lambda^{\vee}-\mu$:} We grow this final face
using $\emptyset-\nu$ and $\nu-\lambda^{\vee}$, but need to make two
consistency checks. First the edge $\emptyset-\mu$ is 
${\tt infusion}_1(T_{\nu},T_{\lambda^{\vee}/\nu})$ which clearly
is $T_{\mu}$. 

It remains to check that 
\begin{equation}
\label{eqn:finalcheck}
{\tt infusion}_2(T_{\nu},T_{\lambda^{\vee}/\nu})=T_{\lambda^{\vee}/\mu}.
\end{equation}
(Notice that ${\tt infusion}_2(T_{\nu},T_{\lambda^{\vee}/\nu})$
is a filling of $\lambda^{\vee}/\mu$ that rectifies to $T_{\nu}$,
just as $T_{\lambda^{\vee}/\mu}$ does. However it is not clear \emph{a priori}
that they are the same.)

To prove (\ref{eqn:finalcheck}), we need a definition.
For tableaux $A$ and $B$ of respective 
(skew) shapes $\alpha$ and $\gamma/\alpha$ 
let $A\star B$ be their concatenation as a 
(nonstandard) tableau. If $C$ is a tableau of shape 
$\Lambda/\gamma$ and $\alpha$ is a straight shape, $A\star B\star C$ is a
{\bf layered tableau} of shape $\Lambda$.

\begin{Example}
\label{exa:mainone}
Let $\alpha=(2,1)$ and $\gamma=(4,2,1)$. We have
\[A=\tableau{{1}&{2}\\{3}}, \ \ B=\tableau{{\ }&{\ }&{{\bf 2} }&
{\bf 3}\\{\ }&{{\bf 1} }\\{\bf 4}}, \ \  
C=\tableau{{\ }&{\ }&{ \ }&{\ }\\{\ }&{\ }&{\underline 2 }&{\underline 3}\\{\ }&{\underline 1}&{\underline 4}&{\underline 5}}, \mbox{ and }
A\star B\star C=\tableau{{1 }&{2 }&{{\bf 2} }&{{\bf 3} }\\{3}&{{\bf 1} }&{{\underline 2} }&{\underline 3}\\{{\bf 4} }&{\underline 1}&{\underline 4}&{\underline 5}}.
\]
We have used boldface and underlining to distinguish the entries from 
$A,B$ and $C$.
\end{Example}

Let ${\mathcal I}_1$ and ${\mathcal I}_2$ be operators on layered tableaux
defined by
\[{\mathcal I}_1:A\star B\star C\mapsto  
{\tt infusion}_1(A,B)\star {\tt infusion}_2(A,B)\star C\]
and
\[{\mathcal I}_2:A\star B\star C\mapsto  
A\star{\tt infusion}_1(B,C)\star {\tt infusion}_2(B,C).\]
The infusion involution says ${\mathcal I}_1^2$ and ${\mathcal I}_2^2$
are the identity operator. In fact, the following crucial
``braid identity'' holds,
showing ${\mathcal I}_1$ and ${\mathcal I}_2$ generate a 
representation of $S_3$:
\begin{Proposition}
\label{prop:braid}
${\mathcal I}_1 \circ {\mathcal I}_2 \circ {\mathcal I}_1 = 
{\mathcal I}_2 \circ {\mathcal I}_1\circ {\mathcal I}_2$.
\end{Proposition}

In view of the Proposition, (\ref{eqn:finalcheck}) 
follows by setting
$\alpha=\lambda$, $\gamma=\nu^{\vee}$,
$A=T_{\lambda}, B=T_{\lambda^{\vee}/\nu}$ and $C={\widetilde T}_{\nu}$,
since the assertion merely says the ``middle'' tableau in
\begin{equation}\label{eqn1}{\mathcal I}_1\circ {\mathcal I}_2\circ {\mathcal I}_1\circ {\mathcal I}_2(T_{\lambda}\star T_{\nu^{\vee}/\lambda}\star {\widetilde T}_{\nu})
 \mbox{ and }{\mathcal I}_2\circ {\mathcal I}_1(T_{\lambda}\star T_{\nu^{\vee}/\lambda}\star {\widetilde T}_{\nu})
\end{equation}
are the same, whereas we even have equality of the two layered tableaux.

\noindent
\emph{Proof of Proposition~\ref{prop:braid}:} 
By the Lemma it follows that
\begin{equation}
\label{eqn:firsttoshow}
{\widetilde C}\star{\overline B}\star {\widetilde A}:={\mathcal I}_1 \circ {\mathcal I}_2 \circ {\mathcal I}_1(A\star B\star C), \mbox{ and }
\end{equation}
\begin{equation}
\label{eqn:secondtoshow}
{\widetilde C}\star{\hat B}\star {\widetilde A}:={\mathcal I}_2 \circ {\mathcal I}_1 \circ {\mathcal I}_2(A\star B\star C),
\end{equation}
where the shapes of ${\widetilde C},
{\overline B}$ and ${\widetilde A}$ are respectively
the 180 degree rotations of the shapes of $A$, $B$ and $C$, and we know
the fillings of ${\widetilde C}$ and ${\widetilde A}$
in terms of evacuation. We thus also know the shapes
of ${\overline B}$ and ${\hat B}$ are the same. 
It remains to show $\overline B$ equals $\hat B$.

We first study (\ref{eqn:firsttoshow}). Let $B'=\infusion_1(A,B)$, and 
$A'=\infusion_2(A,B)$.    
Now
\begin{equation}
\label{eqn:121_case}
\mathcal I_1 \circ \mathcal I_2(B'\star A'\star C)=
\widetilde C\star \overline B \star \widetilde A,
\end{equation}
and the  skew tableau ${\overline B}\star{\widetilde A}$ 
is obtained by treating 
$B'\star A'$ as a \emph{single standard tableau} by valuing an entry $i$
of $A'$ as $|B'|+i$, evacuating,
rotating the result 180 degrees and finally sending the entry $i$ 
from $A'$ (respectively $B'$) to $|A'|-i+1$ (respectively $|B'|-i+1$).

\begin{Example}
\label{exa:121exa}
Continuing the previous example, 
\[B'\star A'=\tableau{{\bf 1}&{\bf 2}&{\bf 3}&{2}\\{\bf 4}&{3}\\{1}}\mbox{ and } {\tt evac}(B'\star A')=A''\star B''=\tableau{{1}&{3}&{\bf 3}&{\bf 4}\\{2}&{\bf 1}\\{\bf 2}}
\mapsto{\overline B}\star{\widetilde A}=\tableau{{\ }&{\ }&{ \ }&{\bf 3 }\\
{\ }&{\ }&{ \bf 4 }&{2 }\\ {\bf 1 }&{\bf 2 }&{1 }&{3 }}
\]
\end{Example}

Let us focus on $\evac(B'\star A')=A''\star B''$ and momentarily ignore
the subsequent rotation and complementation of entries. Fomin
has shown in \cite[Appendix~2]{Stanley} that a fruitful way to think of
${\tt evac}(X)$ is with triangular growth diagrams obtained by placing
the shape chain of $X$, then $\Delta(X),\Delta^2(X),\ldots$ on top of one 
another (slanted left to right); see Figure~\ref{fig:evacgrowth}.

\begin{figure}[h]
\unitlength=.28mm
\begin{picture}(320,160)
\thinlines
\put(20,10){\line(1,1){140}}
\put(160,150){\line(1,-1){140}}
\put(60,10){\line(1,1){120}}
\put(100,10){\line(1,1){100}}
\put(140,10){\line(1,1){80}}
\thicklines
\put(180,10){\line(1,1){60}}
\thinlines
\put(220,10){\line(1,1){40}}
\put(260,10){\line(1,1){20}}
\put(260,10){\line(-1,1){120}}
\put(220,10){\line(-1,1){100}}
\thicklines
\put(180,10){\line(-1,1){80}}
\thinlines
\put(140,10){\line(-1,1){60}}
\put(100,10){\line(-1,1){40}}
\put(60,10){\line(-1,1){20}}
\put(15,0){$\emptyset$}
\put(55,0){$\emptyset$}
\put(95,0){$\emptyset$}
\put(135,0){$\emptyset$}
\put(175,0){$\emptyset$}
\put(215,0){$\emptyset$}
\put(255,0){$\emptyset$}
\put(295,0){$\emptyset$}
\put(90,90){$\mu$}
\put(157,155){$\gamma$}
\put(242,70){$\alpha$}
\put(45,55){$B'$}
\put(140,55){${\tt evac}(B')$}
\put(275,45){$A''$}
\put(200,45){$A$}
\put(115,125){$A'$}
\put(208,110){$B''$}
\end{picture}
\caption{\label{fig:evacgrowth} A triangular growth diagram to compute
${\tt evac}(B'\star A')$}
\end{figure}
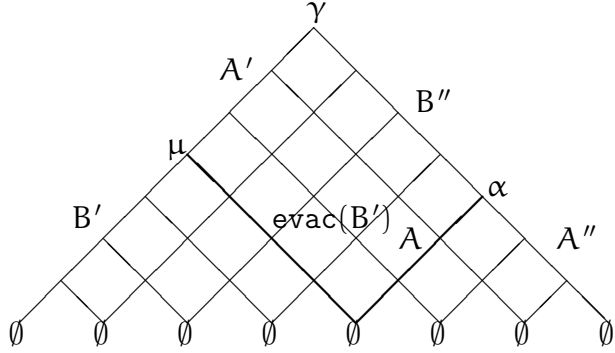
We begin with the data $B'\star A'$ along the left side of the
diagram. This is given as a concatenation of two shape chains, one from 
$\emptyset-\mu$ and then one from $\mu-\gamma$ where $\mu$ is the
shape of $B'$. Applying (F1) and (F2) from Section~1, we obtain
$A''\star B''={\tt evac}(B'\star A')$, given along the righthand side as
a pair of chains connected at $\alpha$. It follows from the construction of
growth diagrams that the thick lines represent
${\tt evac}(B')$ and $A=\evac(A'')$ (the latter being a consequence
of the Lemma).

These thick lines together with the vertex $\gamma$ define a
rectangular growth diagram, so:
\begin{equation}
\label{eqn:1conclusion}
{\tt infusion}(A,B'')=({\tt evac}(B'),A').
\end{equation}
In particular, the shape of ${\tt rectification}(B'')=\mu$.
On the other hand, by definition
\begin{equation}
\label{eqn:2conclusion} 
{\tt infusion}(A,B)=(B',A'),
\end{equation}
and $B$ rectifies to a tableau of shape $\mu$ also.

By (\ref{eqn:1conclusion}) and (\ref{eqn:2conclusion}) combined 
with the aforementioned results of 
\cite{Haiman:DE} we see $B$ and $B''$ are dual equivalent, as they 
both are obtained by an application of the same sequence of 
slides (encoded by $A'$) to a pair of tableaux of the same straight shape.
Moreover, by (\ref{eqn:1conclusion}) $B''$ is jeu de taquin
equivalent to ${\tt evac}(B')={\tt evac}({\tt rectification}(B))$. Carrying out  
the analogous ``reverse'' 
analysis on ${\mathcal I}_2 \circ {\mathcal I}_1 \circ {\mathcal I}_2(A\star B\star C)$ 
we let $C^{\circ}={\tt infusion}_1(B,C)$ and
$B^{\circ}={\tt infusion}_2(B,C)$ and study
${\mathcal I}_2\circ {\mathcal I}_1(A\star C^{\circ}\star B^{\circ})
={\widetilde C}\star {\hat B}\star {\widetilde A}$.
Parallel to (\ref{eqn:121_case}) we
consider the ``reverse evacuation'' $B^{\circ\circ}\star C^{\circ\circ}$ 
of $C^{\circ}\star B^{\circ}$ using \emph{reverse} jeu de
taquin slides. Then we similarly conclude $B^{\circ\circ}$ is dual equivalent to $B$
(and thus $B''$). Also $B^{\circ\circ}$ is 
jeu de taquin equivalent to the reverse evacuation of ${\tt revrectification}(B)$. 
Thus by the lemma, $B''$ and $B^{\circ\circ}$ are jeu de taquin equivalent. Hence
$B''=B^{\circ\circ}$ and so ${\overline B}={\hat B}$, as required.

\begin{Example}
Revisiting Example~\ref{exa:mainone},
$C^{\circ}\star B^{\circ}=\tableau{&&{\underline 2}&{\underline 3}\\&{\underline 4}&{\underline 5}&{\bf 3}\\{\underline 1}&{\bf 1}&{\bf 2}&{\bf 4}}$ with reverse evacuation 
$\tableau{&&{\bf 3}&{\bf 4}\\&{\bf 1}&{\underline 1}&{\underline 2}\\{\bf 2}&{\underline 3}&{\underline 4}&{\underline 5}}$, and $B^{\circ\circ}$ (boldface in the latter tableau) 
is also $B''$ from Example~\ref{exa:121exa}. Rotation and complementation gives ${\hat B}$, which is ${\overline B}$ from Example~\ref{exa:121exa}, as desired. \qed
\end{Example}

\section{An extended example of the main theorem}
Let 
$\Lambda=3\times 4=\tableau{{\ }&{ \ }&{\ }&{ \ }\\ {\ }&{ \ }&{\ }&{ \ }\\ {\ }&{ \ }&{\ }&{ \ }},\lambda=(2,1)=\tableau{{\ }&{ \ }\\{\ }}, \mu=(3,1)=\tableau{{\ }&{\ }&{\ }\\{\ }}, \nu=(3,2)=\tableau{{\ }&{\ }&{\ }\\ {\ }&{\ }}.$
Therefore $\lambda^{\vee}=(4,3,2), \mu^{\vee}=(4,3,1)$ and $\nu^{\vee}=(4,2,1)$. Also
\[T_{\lambda}=\tableau{{1 }&{ 2 }\\{3 }}, T_{\mu}=\tableau{{1 }&{2 }&{3 }\\{4 }}, T_{\nu}=\tableau{{1 }&{2 }&{3 }\\ {4 }&{5 }}, \ 
{\widetilde T}_{\lambda}=\tableau{{\ }&{ \ }&{\ }&{ \ }\\ {\ }&{ \ }&{\ }&{ 1 }\\ {\ }&{ \ }&{2 }&{ 3 }},\ 
{\widetilde T}_{\mu}=\tableau{{\ }&{ \ }&{\ }&{ \ }\\ {\ }&{ \ }&{\ }&{ 3 }\\ {\ }&{ 1 }&{2 }&{ 4 }}, \mbox{ and }
{\widetilde T}_{\nu}=\tableau{{\ }&{ \ }&{\ }&{ \ }\\ {\ }&{ \ }&{2 }&{ 3 }\\ {\ }&{ 1 }&{4 }&{ 5 }}.
\]

Here $C_{\lambda,\mu,\nu}=1$, and we now give the unique carton filling.
We begin with 
$T_{\nu^{\vee}/\lambda}=\tableau{{\ }&{\ }&{2}&{3}\\{\ }&{4}\\{1}}$.
Then we have the following sides of the carton, described as growth diagrams
with the obvious identifications of boundaries
(the partitions correspond to shapes placed on the vertices
of the carton and the diagrams have been oriented to be consistent with Figure~\ref{fig:growth}):

\begin{table}[h]
\begin{center}
\begin{tabular}{|l|l|l|l|l|}
\hline
$(2,1)=\lambda$ & $(2,1,1)$ & $(3,1,1)$ & $(4,1,1)$ & $(4,2,1)=\nu^{\vee}$\\\hline 
$(2)$ & $(2,1)$ & $(3,1)$ & $(4,1)$ & $(4,2)$\\ \hline
$(1)$ & $(1,1)$ & $(2,1)$ & $(3,1)$ & $(3,2)$\\ \hline
$\emptyset$ & $(1)$ & $(2)$ & $(3)$ & $(3,1)=\mu$ \\ \hline
\end{tabular}
\end{center}
\caption{$\emptyset-\mu-\nu^{\vee}-\lambda$}
\end{table}

\begin{table}[h]
\begin{center}
\begin{tabular}{|l|l|l|l|l|}
\hline
$(4,3,1)=\mu^{\vee}$ & $(4,3,2)$ & $(4,3,3)$ & $(4,4,3)$ & $(4,4,4)=\Lambda$\\\hline 
$(4,2,1)$ & $(4,2,2)$ & $(4,3,2)$ & $(4,4,2)$ & $(4,4,3)$\\ \hline
$(4,2)$ & $(4,2,1)$ & $(4,3,1)$ & $(4,4,1)$ & $(4,4,2)$\\ \hline
$(3,2)$ & $(3,2,1)$ & $(3,3,1)$ & $(4,3,1)$ & $(4,3,2)$ \\ \hline
$(2,2)$ & $(2,2,1)$ & $(3,2,1)$ & $(4,2,1)$ & $(4,2,2)$ \\ \hline
$(2,1)=\lambda$ & $(2,1,1)$ & $(3,1,1)$ & $(4,1,1)$ & $(4,2,1)=\nu^{\vee}$ \\ \hline
\end{tabular}
\end{center}
\caption{$\lambda-\mu^{\vee}-\Lambda-\nu^{\vee}$}
\end{table}

\begin{table}[h]
\begin{center}
\begin{tabular}{|l|l|l|l|l|l|}
\hline
$(4,3,1)=\mu^{\vee}$ & $(4,2,1)$ & $(4,2)$ & $(3,2)$ & $(2,2)$ & $(2,1)=\lambda$\\ \hline
$(4,2,1)$ & $(4,1,1)$ & $(4,1)$ & $(3,1)$ & $(2,1)$ & $(2)$\\ \hline
$(3,2,1)$ & $(3,1,1)$ & $(3,1)$ & $(2,1)$ & $(1,1)$ & $(1)$\\ \hline
$(3,2)=\nu$ & $(3,1)$ & $(3)$ & $(2)$ & $(1)$ & $\emptyset$\\ \hline
\end{tabular}
\end{center}
\caption{$\emptyset-\nu-\mu^{\vee}-\lambda$}
\end{table}

\begin{table}[h]
\begin{center}
\begin{tabular}{|l|l|l|l|l|}
\hline
$(4,3,1)=\mu^{\vee}$ & $(4,3,2)$ & $(4,3,3)$ & $(4,4,3)$ & $(4,4,4)=\Lambda$\\\hline 
$(4,2,1)$ & $(4,2,2)$ & $(4,3,2)$ & $(4,4,2)$ & $(4,4,3)$\\ \hline
$(3,2,1)$ & $(3,2,2)$ & $(3,3,2)$ & $(4,3,2)$ & $(4,4,2)$\\ \hline
$(3,2)=\nu$ & $(3,2,1)$ & $(3,3,1)$ & $(4,3,1)$ & $(4,3,2)=\lambda^{\vee}$ \\ \hline
\end{tabular}
\end{center}
\caption{$\nu-\lambda^{\vee}-\Lambda-\mu^{\vee}$}
\end{table}

\begin{table}[h]
\begin{center}
\begin{tabular}{|l|l|l|l|l|l|}
\hline
$(4,4,4)=\Lambda$ & $(4,4,3)$ & $(4,4,2)$ & $(4,3,2)$ & $(4,2,2)$ & $(4,2,1)=\nu^{\vee}$\\ \hline
$(4,4,3)$ & $(4,4,2)$ & $(4,4,1)$ & $(4,3,1)$ & $(4,2,1)$ & $(4,2)$\\ \hline
$(4,3,3)$ & $(4,3,2)$ & $(4,3,1)$ & $(3,3,1)$ & $(3,2,1)$ & $(3,2)$\\ \hline
$(4,3,2)=\lambda^{\vee}$ & $(4,2,2)$ & $(4,2,1)$ & $(3,2,1)$ & $(3,1,1)$ & $(3,1)=\mu$\\ \hline
\end{tabular}
\end{center}
\caption{$\lambda^{\vee}-\Lambda-\nu^{\vee}-\mu$}
\end{table}

\begin{table}[h]
\begin{center}
\begin{tabular}{|l|l|l|l|l|}
\hline
$(3,2)=\nu$ & $(3,2,1)$ & $(3,3,1)$ & $(4,3,1)$ & $(4,3,2)=\lambda^{\vee}$\\\hline 
$(3,1)$ & $(3,1,1)$ & $(3,2,1)$ & $(4,2,1)$ & $(4,2,2)$\\ \hline
$(3)$ & $(3,1)$ & $(3,2)$ & $(4,2)$ & $(4,2,1)$\\ \hline
$(2)$ & $(2,1)$ & $(2,2)$ & $(3,2)$ & $(3,2,1)$\\ \hline
$(1)$ & $(1,1)$ & $(2,1)$ & $(3,1)$ & $(3,1,1)$\\ \hline
$\emptyset$ & $(1)$ & $(2)$ & $(3)$ & $(3,1)=\mu$ \\ \hline
\end{tabular}
\end{center}
\caption{$\emptyset-\nu-\lambda^{\vee}-\mu$}
\end{table}

\section*{Acknowledgments}
HT was partially supported by an NSERC Discovery Grant and
AY by NSF grant 0601010. 
This work was partially completed while HT was a visitor at the Centre de
Recherches Math\'ematiques and the University of Minnesota, 
and while AY was at the Fields Institute and the University of
Michigan, Ann Arbor. We thank the Banff International
Research Station for a common visit during the ``Schubert calculus and Schubert geometry'' workshop organized by Jim Carrell and Frank Sottile.
We also thank Sergey Fomin, Allen Knutson, 
Ezra Miller, Vic Reiner, Luis Serrano, David Speyer and Dennis Stanton 
for helpful discussions.

\end{document}